\documentclass[10pt]{amsart}

\usepackage{graphicx}
\usepackage{verbatim}
\usepackage[latin1]{inputenc}
\usepackage[T1]{fontenc}
\usepackage{amsthm}
\usepackage{amsmath}
\usepackage{amssymb}
\usepackage{relsize}
\usepackage[pagebackref,colorlinks=true,linkcolor=blue,urlcolor=blue,citecolor=blue]{hyperref}

\usepackage [all] {xy}
\usepackage{mathrsfs}

\newtheorem{theorem}{Theorem}[section]

\newtheorem{lemma}[theorem]{Lemma}
\newtheorem{proposition}[theorem]{Proposition}
\newtheorem{definition}[theorem]{Definition}

\theoremstyle{definition}
\newtheorem{remark}[theorem]{Remark}

\newtheorem{construction}[theorem]{Construction}
\newtheorem{notation}[theorem]{Notation}

\numberwithin{equation}{section}

\newcommand{\cat}{\textbf{C}}
\newcommand{\cont}{\textbf{Contr}}

\newcommand{\mor}{\text{Mor}}

\newcommand{\homom}{\operatorname{Hom}}
\newcommand{\id}{\text{Id}}

\newcommand{\semicosimplicial}[3]{\[\xymatrix {#1 \ar@<-.3ex>[r] \ar@<.4ex>[r] & #2 \ar@<-.7ex>[r] \ar@<.7ex>[r]\ar[r] & #3 \ar@<-1ex>[r] \ar@<.4ex>[r]\ar@<1.1ex>[r] \ar@<-.3ex>[r] & \dots}\]}
\newcommand{\namesemicosimplicial}[4]{\[\xymatrix {#1: &#2 \ar@<-.3ex>[r] \ar@<.4ex>[r] & #3 \ar@<-.7ex>[r] \ar@<.7ex>[r]\ar[r] & #4 \ar@<-1ex>[r] \ar@<.4ex>[r]\ar@<1.1ex>[r] \ar@<-.3ex>[r] & \dots}\]}
\newcommand{\namesemicosimplicialbis}[3]{\[\xymatrix {#1: &#2 \ar@<-.3ex>[r] \ar@<.4ex>[r] & #3 \ar@<-.7ex>[r] \ar@<.7ex>[r]\ar[r] & \dots}\]}
\newcommand{\semicosimplicialbis}[2]{\[\xymatrix{\text{Hilb}^Z_X(A) = \text{Tot}( #1 \ar<-.3ex>[r] \ar@<.4ex>[r] & #2 \ar@<-.7ex>[r] \ar@<.7ex>[r]\ar[r] & \dots )}\]}
\newcommand{\contr}[5]{\xymatrix{#1 \ar@<.5ex>[r]^{#3}&#2\ar@(ul,ur)^{#4}\ar@<.5ex>[l]^{#5}}}

\newcommand{\fK}{\mathbb K}

\newcommand{\sC}{\mathcal{C}}

\author{Luigi Lunardon}
\title{Some remarks on Dupont contraction}

\begin{document}
\begin{abstract}
We present an alternative equivalent description of
Dupont's simplicial contraction: it is an explicit example of a simplicial contraction between the simplicial differential graded algebra of polynomial differential forms on  standard simplices and the space of Whitney elementary forms.  
\end{abstract}

\maketitle

\section{Introduction}
In \cite{dup} Dupont gave an explicit description of  simplicial contraction from the simplicial differential graded algebra $\Omega_{\bullet}$ of polynomial differential forms on the affine standard simplices to the subspace of Whitney elementary forms, a simplicial finite dimensional differential graded vector subspace of the former.

More precisely, the Dupont contraction is a morphism of simplicial abelian groups $h\colon \Omega_\bullet\to \Omega_\bullet$ such that $dh+hd=r-\id$, where $r$ is the classical Whitney's  retraction of 
$\Omega_\bullet$ onto the subspace of elementary forms: in this paper we recall 
the general notion of contraction  in Section~\ref{c1} and we describe the simplicial map  $h$ in Section~\ref{c11}.

Although Dupont  contraction can be used to give alternative proofs of  some classical results, such as the polynomial De Rham's theorem (\cite[Theorem 2.2]{bg},  \cite{dup2} and \cite[Theorem 10.15]{rht}), their most relevant use is given in combination with homological perturbation theory \cite{hk} and homotopy transfer of $\infty$-structures (see e.g. \cite{bm,f,ks}).
For instance, Dupont's contraction was used in \cite{cg} to construct a \emph{canonical} $C_\infty$ structure on the 
normalized cochain complex of a cosimplicial commutative algebra over a field of characteristic 0. 
Similarly,  in \cite{fmm} the authors used it to induce a \emph{canonical} $L_\infty$ structure on the normalized cochain complex of a cosimplicial differential graded Lie algebra: in the same paper some explicit computation is done in the particular case of cosimplicial Lie algebra, while the particular case of a single morphism of differential graded Lie algebras (considered as a cosimplicial object via Kan extension) was previously considered and deeply investigated in \cite{fm}. It is also worth to mention the application of Dupont's contraction to Hodge theory of complex algebraic varieties \cite{navarro}.

Dupont's Theorem and the homotopy transfer theorem are also key tools in \cite{getzler} and \cite{b}. In these two papers the authors study the Deligne $\infty$-groupoid associated to an $L_\infty$-algebra, its relation with the Maurer-Cartan elements of that algebra and the behaviour of the Deligne $\infty$-groupoid under totalization and homotopy limits. In particular Dupont's contraction is used to construct a Kan complex that is quasi-isomorphic to the simplicial set of Maurer-Cartan elements.

The original construction by Dupont provides a family of maps which is really hard to compute. 
An apparently different simplicial contraction $k\colon \Omega_\bullet\to \Omega_\bullet$
with the same key properties of Dupont's contraction, but 
somewhat easier to handle, was proposed by  M. Manetti  during a cycle of seminars on deformation theory given at Roma in 2011, 
leaving  unsettled the question whether $k=h$. 

The main result of this paper is to give a positive answer to the above question, and then 
to give an alternative equivalent definition of Dupont's contraction.
In Section \ref{c2} we describe the map  $k$ and we reproduce Manetti's (unpublished) proof that 
it is indeed a simplicial object in the category of contractions. Finally, in Section~\ref{c3} we prove the equality 
$k=h$.

\subsection*{Acknowledgement}
I would like to thanks Prof. M. Manetti for his help during the (slow) preparation of this paper.
This work was supported by the Engineering and Physical Sciences Research Council [EP/L015234/1]. The EPSRC Centre for Doctoral Training in Geometry and Number Theory (The London School of Geometry and Number Theory), University College London.

\section{Simplicial contraction}\label{c1}
In this section we describe the category of contractions of DG-vector spaces and we recall the definition of simplicial and cosimplicial objects in any given category.

Let $\fK$ be a field of characteristic $0$, a DG-vector space over $\fK$ is a graded vector space endowed with a linear map $d$ of degree $1$ such that $d^2 = 0$.
\begin{definition}\label{con}
A contraction of DG-vector spaces is a diagram 
\[\contr{M}{N}{i}{h}{\pi},\]
with $M$ and $N$ two DG-vector spaces over $\fK$, $h\in\homom_\fK^{-1}(N, N)$ and $i,\;\pi$ two morphisms of DG-vector spaces. Moreover, we require the following relations:
\[\pi i = \id_M, \;\quad\quad i\pi - \id_N = d_Nh + hd_N.\]
\end{definition}
\begin{remark}
The maps $\pi$ and $i$ are respectively injective and surjective, since $\pi i = \id_M$. Moreover, it follows from the relation $ i\pi - \id_N = d_Nh + hd_N$ that they are both quasi-isomorphisms.
\end{remark}
\begin{remark}\label{g}
Suppose also that the additional conditions $h^2 = \pi h = 0$ hold. From the identity $ i\pi - \id_N = d_Nh + hd_N$ we obtain the identities:
\[ - h = hd_N h;\quad\quad hi\pi - h = hd_N h.\]
It follows that $h i \pi = 0$ and since $\pi$ is surjective, then $h i = 0.$ Similarly the conditions $hi = h^2 = 0$ imply $\pi h = 0.$ The conditions $h^2 = \pi h = hi = 0$ are called side conditions.
\end{remark}

\begin{definition}\label{morcon}
A morphism of contractions is a commutative diagram
\[\xymatrix{N\ar@(ul,ur)^{h}\ar@<.5ex>[d]^{\pi}\ar[r]^f & B\ar@(ul,ur)^{k}\ar@<.5ex>[d]^{p}\\
M\ar@<.5ex>[u]^{i}\ar[r]^{\widehat f}&A\ar@<.5ex>[u]^{j}}\]
where $f\colon N\to B$ is a morphism of DG-vector spaces such that $fh = kf.$ We denote by $\widehat f \colon M\to A$ the map $\widehat f = pfi$. 
\end{definition}
\begin{remark}
This definition of morphism doesn't seem natural. However we get the following identities:
\[\begin{split}j\widehat f &= jpfi = (\id_B + d_Bk + kd_B)fi = fi + f(d_Nh + hd_N)i \\
& = fi + f(i\pi - \id_N)i = fi,\\
\widehat f \pi &= pfi\pi = pf + pf(d_N h + h d_N) = pf + p(d_B k + k d_B)f  \\
&=pf + p(jp - \id_B)f = pf.\end{split}\]
Using these these two identities it follows that the following diagrams commute:
\[\xymatrix{M\ar[r]^{\widehat f}\ar[d]^i &A\ar[d]^j\\N\ar[r]^f&B}\quad\quad\quad\quad \xymatrix{N\ar[r]^{f}\ar[d]^\pi &B\ar[d]^p\\M\ar[r]^{\widehat f}&A}\]
As a consequence our notion of morphism of contractions is compatible with a couple of morphism $f\colon N\to B$ and $g\colon M\to A$ commuting with every square.
\end{remark}
The category of contractions of DG-vector spaces over $\fK$ is denoted by $\cont$.
\begin{definition}
We denote with $\Delta$ the category of finite ordinals. The objects of this category are the finite ordered sets $[n] = \{0 < \dots < n\}$ and its morphisms are the non decreasing maps. 
A special role in this category is played by face maps, which are defined as:
\[\delta_k\colon [n - 1] \to [n];\quad\quad \delta_k(x) = \left\{\begin{array}{lcc} x\quad\quad\quad\quad\text{if $p < k$}\\x + 1\quad\quad\text{ if $p\geq k$}\end{array}\right .,\quad k = 0, \dots, n. \]
\end{definition}
\begin{notation}
We denote with $I(n, m) \subset \mor_\Delta([n], [m])$ the subset of injective, and hence strictly monotone, maps.
\end{notation}


\begin{definition}
A cosimplicial object in a category $\cat$  is a functor $F\colon \Delta\to \cat$; a simplicial object in $\cat$ is a functor $F\colon \Delta^{op}\to \cat.$ 
\end{definition}

Dupont (\cite{dup}, Chapter 2) proposed an explicit construction of a simplicial object in $\cont$.

\begin{remark}
The notion of contraction has a few slight variants in literature. In this paper we follow \cite{cg} and \cite{getzler}.
The original definition given by Eilenberg and Mac Lane in \cite{eml} requires also the side conditions $hi = \pi h = 0$. 
In \cite{ls} the object described in Definition \ref{con} is called a \emph{strong deformation data}, and to be a contraction the condition $h^2 = 0$ is required. 

The conditions $hi = \pi h = h^2 = 0$ are almost granted: given $i ,\pi$ and $h$ as in Definition \ref{con}, then we can replace $h$ with $h_1 = (dh + h d) h(dh + h d)$; we still have a contraction, but now this contraction satisfies $h_1 i = \pi h_1 = 0.$ Replacing $h_1$ with $h_2 = -h_1dh_1$ it is again a contraction and now it satisfies $\pi h_2 = h_2i =h_2^2 = 0$

\end{remark}






\section{Dupont's simplicial contraction}\label{c11}
In this section we describe the simplicial contraction which Dupont suggested in \cite{dup}. The proof that it is actually a contraction is in Section \ref{c2} and Section \ref{c3}. 

\begin{definition}
The affine standard $n$-simplex on $\fK$ is the set:
\[\Delta^n_\fK\ = \{(x_0, \dots, x_n)\in \fK^{n + 1} \text{ such that } x_0 + \dots + x_n = 1\}.\]
The vertices of $\Delta_{\fK}^n$ are the points $e_i\in\Delta_\fK^n$:
\[e_0 =  (1,0,\dots, 0), \quad e_1 = (0, 1, 0, \dots, 0), \quad\dots\quad,\quad e_n = (0,\dots, 0, 1).\]
\end{definition}
The \emph{cosimplicial affine space $\Delta^\bullet_{\fK}$} is the functor which associate to each set $[n]$ the affine standard $n$-simplex $\Delta^n_\fK$ and to each non decreasing map $f\colon[n]\to[m]$ the affine map 
\[f\colon \Delta_\fK^n\to\Delta_\fK^m, \quad\quad\quad f(e_i) = e_{f(i)}.\]
\begin{definition}
The DG-vector space of polynomial differential forms on the affine standard $n$-simplex is:
\[\Omega_n = \bigoplus\limits_{p=0}^n \Omega_n^p= \frac{\fK[x_0, \dots, x_n, dx_0, \dots, dx_n]}{\left(\sum\limits_{k=0}^n x_i - 1, \;\sum\limits_{k=0}^n dx_i\right)}.\]
Here $\Omega_n^p$ denotes the subspace of $p$-forms, which are the elements of degree $p$.
\end{definition}
The simplicial DG-vector space $\Omega_\bullet$ associates to each $[n]$ the DG-vector space $\Omega_n$ and to each map $f\colon[n]\to[m]$ in $\Delta$ the pull-back $f^\ast\colon\Omega_m\to \Omega_n$, induced by the affine map $f\colon \Delta^n_\fK \to \Delta^m_\fK$. 

Next we define a finite dimensional vector subspace $\sC_n\subset\Omega_n,$ called \emph{the space of Whitney elementary forms}. As a consequence of Proposition \ref{wff} it follows that $\sC_n$ is closed under derivation and thus it is a DG-vector subspace of $\Omega_n$.
\begin{definition}[Whitney, \cite{whit}]\label{wef}
Fix $f\colon [m] \to [n]$ a morphism in $\Delta$. The Whitney elementary form associated to $f$ is the $m$-form:
\[\omega_f = m! \sum\limits_{i = 0}^{m} (-1)^i x_{f(i)} dx_{f(0)}\wedge \dots\wedge \widehat{dx_{f(i)}}\wedge \dots\wedge dx_{f(m)} \in \Omega_n^m.\]
We denote with $\sC_n$ the vector space spanned by Whitney elementary forms.
\end{definition}

\begin{proposition}\label{wff}
Let $f\colon [n]\to [m]$ a morphism in $\Delta.$ The followings hold:
\begin{enumerate}
\item If $f$ is injective $ f^\ast\omega_f = n! dx_1\wedge \dots \wedge dx_n;$ otherwise $\omega_f = 0$;
\item If $f$ is injective for every $g\colon [p] \to [m]$ we have $g^\ast\omega_f = \sum\limits_{\{h\colon[n]\to[p], \; gh = f\}}\omega_h;$
\item $d\omega_f = \sum\limits_k (- 1)^k \sum\limits_{\{g\colon[n + 1] \to [m], \;g\delta_k = f\}}\omega_g.$
\end{enumerate}
In particular $\sC_n$ is a DG-vector subspace of $\Omega_n$ and $\sC_\bullet$ is a simplicial DG-vector subspace of $\Omega_\bullet.$
\end{proposition}
\begin{proof}
Denote by $\omega_{i_0,\dots, i_n}$ the differential form:
\[\omega_{i_0,\dots, i_n} = n!\sum\limits_{k = 0}^n (-1)^k x_{i_k} dx_{i_0}\wedge\dots\wedge\widehat{dx_{ik}}\wedge\dots\wedge dx_{i_n}\in\Omega_m^n.\]

\begin{enumerate}
\item Since $\omega_{i_0,\dots, i_n}$ is alternating on indices, then if $f$ is not injective it follows $\omega_f = 0.$ 

Suppose $f$ injective; since we are working on the affine standard simplex we have:
\[\begin{split}f^\ast\omega_f &= n!\sum\limits_{k = 0}^n (-1)^kx_kdx_0\wedge\dots\wedge\widehat{dx_k}\wedge\dots\wedge dx_n \\
&= n!\left(x_0dx_1\wedge\dots\wedge dx_n - \sum\limits_{k = 1}^n (-1)^{2k - 1} x_k dx_1\wedge\dots\wedge dx_n\right ) \\
&= n!\,(x_0+ \dots + x_n) dx_1\wedge\dots\wedge dx_n = n!dx_1\wedge\dots\wedge dx_n.\end{split}\]

\item Consider the family of sets $P_i = \{ j\in[p]| \;g(j) = f(i)\}$ and note that:
\begin{enumerate}
\item since $f$ is injective $P_i\cap P_j = \emptyset$ if $i \neq j$;
\item $g^\ast(x_{f(i)}) = \sum_{j\in P_i} x_j,$ and $g^\ast(dx_{f(i)}) = \sum_{j\in P_i} dx_j$;
\item the functions $h\colon[n]\to[p]$ such that $gh = f$ are in bijection with $P_0\times\dots\times P_n$.
\end{enumerate}

\item To prove the last point first we show that 
\[d\omega_{i_0,\dots, i_n} = (n + 1)! dx_{i_0} \wedge\dots\wedge dx_{i_n} = \sum_{i = 0}^m \omega_{i,i_0,\dots, i_n}.\]
Indeed we have:
\[d\omega_{i_0,\dots, i_n} = n!\sum\limits_{k = 0}^n dx_{i_0}\wedge\dots \wedge dx_{i_k}\wedge\dots\wedge dx_{i_n}=(n + 1)! dx_{i_0} \wedge\dots\wedge dx_{i_n},\]
and for the second equality:
\[\begin{split}\sum\limits_{i = 0}^m\omega_{i,i_0,\dots, i_n} &= (n + 1)!\sum\limits_{i = 0}^m x_idx_{i_0}\wedge\dots\wedge dx_{i_n} - (n + 1)\sum\limits_{i = 0}^m dx_i\wedge\omega_{i_0, \dots, i_n}\\
&= (n + 1)! dx_{i_0} \wedge\dots\wedge dx_{i_n}.\end{split}\]
Taking $f\colon[n]\to[m]$ we can finally see that:
\[\begin{split} d\omega_f = \sum\limits_{i = 0}^m \omega_{i, f(0), \dots, f(n)} &= \sum\limits_{k = 0}^n (- 1)^k \sum\limits_{f(k - 1) < i< f(k)} \omega_{f(0), \dots, f(k -1), i, f(k), \dots, f(n)}\\
&= \sum\limits_{k = 0}^n (-1)^k \sum\limits_{\{g|g\delta_k = f\}}\omega_g.
\end{split}\]
\end{enumerate}
In particular it follows from (1) that $\sC_n$ is a finite dimensional DG-vector space for all $n$, while (2) and (3) imply that $\sC_\bullet$ is a simplicial DG-vector space.
\end{proof}

From Proposition \ref{wff} we obtain, for all $n \geq 0$, the inclusion of DG-vector spaces:
\[\sC_n\xrightarrow{i_n} \Omega_n.\]
We want to extend these inclusions to contractions. 
This means that we want to introduce two families of maps, $\pi_m$ and $h_m$ such that $h_m, \pi_m$ and $i_m$ satisfy the conditions of Definition \ref{con}. Moreover we want this construction to be simplicial.

To define these maps we use integration on affine standard simplices. An axiomatic definition of integration of polynomial differential forms on affine standard simplices and a more detailed discussion on its properties is given in Chapter 10 of \cite{rht}.
The integration map on affine standard simplices is the map
\[\int\limits_{\Delta^n_\fK}\;\;\;\;\colon \Omega_n \to \fK,\]
defined by linearity using the two identities: 
\begin{equation}\label{nneqp}
\mathlarger{\int\limits_{\Delta^n_\fK} \eta} = 0, \quad\quad \text{if } \eta\in\Omega_n^p, \text{ with}\;p \neq n,
\end{equation}
\begin{equation}\label{n=p}
\int\limits_{\Delta^n_\fK}x_0^{k_0}\cdots x_n^{k_n} dx_0\wedge\dots\wedge\widehat{dx_i}\wedge\dots\wedge dx_n = (-1)^i \frac{k_0!\cdots k_n!}{(k_0+ \dots+ k_n + n)!}\,.
\end{equation}
In Construction \ref{h_m} we describe the family of maps $h_m\in \homom^{-1}_{\mathbb{K}}(\Omega_m, \Omega_m)$ defined by Dupont in \cite{dup}.

\begin{notation}
We use the following notation: 
\[\widehat\Delta^n = \{(s, t_0, \dots, t_n)\in \fK^{n + 1}| \; s + t_0 + \dots + t_n = 1\};\]
while the DG-vector space of polynomial differential forms on $\widehat\Delta^n$ is denoted by   $\widehat\Omega_n$.
\end{notation}
\begin{construction}\label{h_m}
Dupont uses $\mathbb R$ as base field, but his construction works in any field of characteristic $0$.
The map $f_j\colon[0] \to [m]$ is defined as $f_j(0) = j$; to this one we associate the map $\widehat{f_j}$ defined as: 
\[\widehat{f_j}\colon \widehat\Delta^0 \times \Delta^m_\fK\to\Delta_\fK^m,\quad\quad \widehat{f_j}((s, t_0), v) = se_j + t_0 v = se_j + (1 - s)v.\]
For any $\eta \in \Omega_m$, since  $s + t_0 = 1$ and $ds + dt_0 = 0$, there are unique forms $\alpha_\eta, \;\beta_\eta \in \Omega_m[s]$ such that $\widehat{f_j}^\ast(\eta) = ds \wedge \alpha_\eta + \beta_\eta$, where $\widehat{f_j}^\ast\colon \widehat\Omega_0\otimes\Omega_m\to \Omega_m$ is the pull-back map. 
First suppose 
\[\alpha_\eta =(1 - s)^a s^b x_0^{k_0}\cdots x_n^{k_n} dx_{c_1}\wedge \dots\wedge dx_{c_l}\] 
then, following Dupont's notation, we define $h_j\in\homom_\fK( \Omega_m,\Omega_m)^{-1}.$ 
\[\begin{split}h_j(\eta) &= \left(\int\limits_0^1(1 - s)^as^b ds\right) x_0^{k_0}\cdots x_n^{k_n}dx_{c_1}\wedge \dots\wedge dx_{c_l} \\&= \left(\int\limits_{\widehat\Delta^0}t_0^as^b dt_0\right) x_0^{k_0}\cdots x_n^{k_n}dx_{c_1}\wedge \dots\wedge dx_{c_l}.\end{split}\]
Now we can extend this by linearity.

For any strictly increasing morphism$f\colon[n] \to [m]$ we define:
\[h_f\in\homom_\fK^{-n-1}( \Omega_m,\Omega_m), \quad\quad h_f = h_{f(n)}\circ \dots \circ h_{f(0)},\]
and $h_m\in\homom_\fK^{-1}(\Omega_m, \Omega_m)$ is:
\[h_m(\eta) = \sum\limits_{n = 0}^m\sum\limits_{f\in I(n, m)} \omega_f\wedge h_f(\eta).\]
\end{construction}
Next we describe  some properties of the maps $h_f$. These results were proved by Dupont in different parts of Chapter 2 of \cite{dup}.
\begin{lemma}\label{proph_f}
Take $f\colon[n] \to [m]$ and $g\colon[m]\to[p]$ then:
\begin{enumerate}
\item \[g^\ast\circ h_{gf} = h_f\circ g^\ast,\]
\item \[[h_f, d](\eta) = h_f(d\eta) + (-1)^ndh_f(\eta) = \int\limits_{\Delta_\fK^n}f^\ast\eta - \sum\limits_{i = 0}^n (-1)^i h_{f\delta_i}(\eta).\]
We use the convention that $h_{f\delta_0}$ is the identity.
\end{enumerate}
\end{lemma}

\begin{theorem}[Dupont, \cite{dup}]\label{dupont}
Consider for each $m\geq 0$ the two operators 
\[\pi_m\in\homom^0_\fK(\Omega_m, \Omega_m),\quad\quad \pi_m(\eta) = \sum\limits_{n = 0}^m\sum\limits_{f\in I (n, m)}\left(\mathlarger{\int_{\Delta^n_\fK}f^\ast\eta}\right)\omega_f,\]
the following diagram is a simplicial contraction
\[\contr{\sC_\bullet}{\Omega_\bullet}{i_\bullet}{h_\bullet}{\pi_\bullet}\]
\end{theorem}

\begin{remark}\label{endc2}
The first definition of Whitney elementary forms  can be found in Section 27 of \cite{whit}; they are defined exactly as those forms $\omega_{i_0, \dots, i_n}$ which appear in the proof of Proposition \ref{wff}. We prefer Definition \ref{wef} due to Point (2) and (3) of Proposition \ref{wff}. 

The notation used for the space of polynomial differential forms is the same of \cite{getzler}. An other common notation present in literature is the one of \cite{rht} - here $\Omega_m$ is denoted with $(A_{PL})_m$.

The family of maps $\{\pi_m\}$ was defined by Whitney in \cite{whit}, Dupont described the family of maps $\{h_m\}$ explicitly in the original proof of Theorem \ref{dupont} given in \cite{dup}. Later Getzler  showed in \cite{getzler} that side conditions $\pi_mh_m = h_m^2 = 0$ (and hence $h_mi_m = 0$) hold.

\end{remark}











\section{The proof of Dupont's Theorem}\label{c2}
In this section we describe a family of maps $k_m$ such that the diagram
\[\contr{\sC_\bullet}{\Omega_\bullet}{i_\bullet}{k_\bullet}{\pi_\bullet}\]
is a simplicial contraction. The family of maps $k_m$ and the proof of this result was shown to us by Manetti at a cycle of seminars at the University ``La Sapienza''.

Recall that $I(n, m)$ is the subset of $\mor_\Delta([n], [m])$ of injective (and hence strictly increasing) morphisms.
\begin{construction}\label{k_m}
Take $f\in I(n, m)$ we define:
\[\widehat f\colon \widehat\Delta^n\times\Delta^m_\fK \to \Delta^m_\fK, \quad\quad ((s, t_0, \dots, t_n), v)\mapsto sv + \sum\limits_{i= 0}^n t_i e_{f(i)}.\]
The operator $k_f\in\homom_\fK^{-n-1}(\Omega_m, \Omega_m)$ is:
\[k_f\colon\Omega_m\xrightarrow{\quad\widehat{f}^\ast\quad}\widehat\Omega_n\otimes\Omega_m\xrightarrow{\mathlarger{\;\;\int\limits_{\widehat\Delta^n}\cdot\;}\otimes\id\quad}\Omega_m,\]
where $\widehat f^\ast\colon \Omega_m\to\widehat\Omega_n\otimes\Omega_m$ is the usual pull-back map.

The map $k_m \in \homom_\fK^{-1}(\Omega_m, \Omega_m)$ is 
\[k_m(\eta) =  \sum\limits_{n = 0}^m\sum\limits_{f\in I(n, m)} \omega_f\wedge k_f(\eta).\]
\end{construction}
The next lemma is exactly Lemma \ref{proph_f}, but with the maps $k_f$ instead of the maps $h_f$.
We don't give a proof of this lemma. This result follows, a posteriori, from Lemma \ref{n=0}, Lemma \ref{case1} and Lemma \ref{case2}.

\begin{lemma}\label{propk_f}
Take $f\colon[n] \to [m]$ and $g\colon[m]\to[p]$ then:
\begin{enumerate}
\item \[g^\ast\circ k_{gf} = k_f\circ g^\ast,\]
\item \[[k_f, d](\eta) = k_f(d\eta) + (-1)^ndk_f(\eta) = \int\limits_{\Delta_\fK^n}f^\ast\eta - \sum\limits_{i = 0}^n (-1)^i k_{f\delta_i}(\eta).\]
We use the convention that $k_{f\delta_0}$ is the identity.
\end{enumerate}
\end{lemma}

The following theorem corresponds to Theorem \ref{dupont}, by replacing the maps $h_m$ with $k_m$.


\begin{theorem}[Dupont, \cite{dup}]\label{manettidupont}
Consider for each $m\geq 0$ the operator $k_m$ of Construction \ref{k_m} and $\pi_m$ of Theorem \ref{dupont}
\[\pi_m\in\homom^0_\fK(\Omega_m, \Omega_m),\quad\quad \pi_m(\eta) = \sum\limits_{n = 0}^m\sum\limits_{f\in I (n, m)}\left(\mathlarger{\int_{\Delta^n_\fK}f^\ast\eta}\right)\omega_f,\]
\begin{enumerate}
\item the operator $\pi_m$ is a projector onto $\sC_m$;
\item the identity $k_md + dk_m = i_m\pi_m - \id_{\Omega_m}$ holds;
\item for every  $p\in\mathbb N$ and every $g\colon [p] \to [m]$ we have $k_pg^\ast = g^\ast k_m.$
\end{enumerate}
\end{theorem}

\begin{proof}
From Point $(1)$ and Point $(2)$ of Proposition \ref{wff} given $f\in I(n,m)$ we have:
\[\int\limits_{\Delta_\fK^n} f^\ast\omega_f = 1,\quad\quad \int\limits_{\Delta_\fK^n} f^\ast\omega_g = 0\quad \text{if } f \neq g, \]
thus $\pi_m$ projects to $\sC_m$.

For every 
$\eta\in \Omega_m$ we have:
\[\begin{split} 
k_m(d\eta) + dk_m\eta 
&=\sum\limits_{n = 0}^m\sum\limits_{f\in I(n, m)} d\omega_f\wedge k_f(\eta) + \omega_f\wedge \left ( (-1)^n dk_f(\eta) + k_f(d\eta)\right )  \\
&=\sum\limits_{n = 0}^m\sum\limits_{f\in I(n, m)} d\omega_f\wedge k_f(\eta) +  \omega_f\wedge \left (\int\limits_{\Delta_\fK^n} f^\ast\eta -\sum\limits_{r = 0}^n (-1)^r k_{f\delta_r}(\eta)\right) \\
&=\sum\limits_{n = 0}^m\sum\limits_{f\in I(n, m)}\left( d\omega_f\wedge k_f(\eta) - \omega_f\wedge \left (\sum\limits_{r = 0}^n (-1)^r k_{f\delta_r}(\eta)\right)\right) + \pi_m(\eta)\,.
\end{split}\]
Since  $k_{f\delta_0} = \id$ and $\sum\limits_{f\in I(0, m)} \omega_f = \sum\limits_{i = 0}^m t_i = 1$ we have:
\[\sum\limits_{f\in I(0,m)} \omega_f\wedge\left(\sum\limits_{r = 0}^0(-1)^r k_{f\delta_r}(\eta)\right) = \sum\limits_{i = 0}^m t_i k_{f\delta_0} (\eta)=\eta.\]
Thus it follows:
\[\begin{split}k_m(d\eta) + dk_m(\eta) - \pi_m(\eta) + \eta  &= \sum\limits_{n = 0}^m\sum\limits_{f\in I(n,m)} d\omega_f \wedge k_f(\eta)\\
&\qquad +\sum\limits_{n = 1}^m\sum\limits_{f\in I(n,m)} -\omega_f\wedge \sum\limits_{r = 0}^n (- 1)^r k_{f\delta_r}(\eta)\,.\end{split}\]
Using the result of Point 3 of Proposition \ref{wff} it is possible to show that the right hand side of this equation vanishes:
\[\begin{split}\sum\limits_{n = 0}^m\sum\limits_{f\in I(n,m)} d\omega_f\wedge k_f(\eta) &= \sum\limits_{n = 0}^{m - 1}\sum\limits_{f\in I(n,m)} d\omega_f\wedge k_f(\eta)\\
&=\sum\limits_{n = 0}^{m - 1}\sum\limits_{f\in I(n,m)}\sum\limits_{r = 0}^n (-1)^r\sum\limits_{\{g| f = g\delta_r\}} \omega_g \wedge k_{g\delta_r}(\eta) \\ 
&+\sum\limits_{n = 1}^{m}\sum\limits_{g\in I(n,m)}\sum\limits_{r = 0}^n (-1)^r \omega_g\wedge k_{g\delta_r}(\eta).\end{split}\]
And thus we proved the identity.

Finally, from Lemma \ref{propk_f} follows that
\[\begin{split} g^\ast k_m(\eta) &= \sum\limits_{n = 0}^m\sum\limits_{f\in I(n,m)} g^\ast(\omega_f) \wedge g^\ast k_f(\eta) = \sum\limits_{n = 0}^m\sum\limits_{f\in I(n,m)} \sum\limits_{\tiny{\begin{array}{c}h\in I(n,p),\\ f=gh\end{array}}} \omega_h\wedge g^\ast k_f(\eta)\\
&=\sum\limits_{n = 0}^m\sum\limits_{h\in I(n,p)}\omega_h \wedge g^\ast k_{gh}(\eta) = \sum\limits_{n = 0}^m\sum\limits_{h\in I(n,p)} \omega_h\wedge k_h(g^\ast\eta) = k_p(g^\ast\eta).\end{split}\]
\end{proof} 

\begin{remark}
This proof of Theorem \ref{manettidupont} was shown to a small audience by Manetti. Dupont in \cite{dup} showed that Lemma \ref{propk_f} holds also for the family of maps $h_m$. Since the proof of Theorem \ref{manettidupont} is based only on Lemma~\ref{propk_f} and on some properties of Whitney elementary forms, the same proof works also for Theorem~\ref{dupont}. 
\end{remark}






\bigskip
\section{Equivalence of the families $h_m$ and $k_m$}\label{c3}

In this section we compare the family of maps $h_m$ of Construction \ref{h_m} and the family of maps of Construction \ref{k_m}. The main result of this section is Theorem \ref{mt}, where we prove that the two families coincide. To make the proof more readable we will split it in many lemmas.

\begin{theorem}\label{mt}
For every $m\in\mathbb N$ we have that $k_m = h_m.$
\end{theorem}
\begin{proof}
From the definition of $h_m$ and $k_m$ if follows that it is enough to prove that for each $f\in I(n,m)$ the identity $k_f = h_f$ holds. 
We proceed by induction on $n$. If $n = 0$ this is Lemma \ref{n=0}; thus suppose $n > 0$ and assume the statement true for every function in $I (p, m)$ with $p < n$. Fix $f\in I(n, m)$, in particular the statement holds for $f|_{[n - 1]},$ and then we have the chain of equality:
\[h_f = h_{f(n)} \circ h_{f|_{[n - 1]}} = h_{f(n)} \circ k_{f|_{[n - 1]}}.\]
The only thing left to prove is $k_f = h_{f(n)} \circ k_{f|_{[n - 1]}}$.
Let $f\colon[n]\to[m]$ be an injective map.
By linearity we can assume that $\eta$ is the q-form 
\[\eta = x_0^{k_0}\cdots x_{f(n)}^0\cdots x_m^{k_m} dx_{c_1}\wedge\dots\wedge dx_{c_q}\]
with $c_1 < c_2 < \dots < c_q$ and $c_i \neq f(n), \text{ for all } i.$

If $q \leq n$ we have $0 = k_{f}(\eta) = h_f(\eta)$, since they are forms of negative degree, hence the equality holds.

Suppose now $q > n$. Let $C \colon= \{c_1, \dots, c_q\}$ and $\text{Im}(f|_{[n-1]})$ if the intersection is such that $|C\cap \text{Im}(f|_{[n-1]})| < n - 1$, then by the same degree argument it follows that $0 = h_{f}(\eta) = k_{f}(\eta)$.
If $|C\cap\text{Im}(f|_{[n-1]})| = n-1$ we are under the hypothesis of Lemma \ref{case1}, so $k_f(\eta) = h_f(\eta).$ If $|C\cap\text{Im}(f|_{[n-1]})| = n$ then we are under the hypothesis of Lemma \ref{case2}, so $k_f(\eta) = h_f(\eta).$
\end{proof}

The next lemmas provide the technicalities behind Theorem \ref{mt}, whose inductive step will be Lemma \ref{n=0} is the inductive base of the proof; Lemma \ref{case1} and Lemma \ref{case2} address the computation of the functions $h_f(\eta)$ and $k_f(\eta)$ in some key cases.

\begin{notation}
When necessary, we use the notation $d_{x_0,\dots, x_n}$ for the differential form $dx_0 \wedge\dots \wedge dx_n.$ 
\end{notation}

\begin{lemma}\label{n=0}
For each integer $0\leq j\leq m,$ consider the map:
\[f_j\colon[0]\to[m]\quad\quad\quad\quad f_j(0) = j.\] 
Then the map $h_j$ of Construction \ref{h_m} and the maps $k_{f_j}$ of Construction \ref{k_m} coincide.
\end{lemma}
\begin{proof}
Take $\eta\in\Omega_m$ and assume without loss of generality $j = 0$ and  
\[\eta = x_0^{k_0}\cdots x_m^{k_m} dx_{c_1}\wedge \dots\wedge dx_{c_l}.\]
The general case will follow by linearity. Call $f$ the map $f_j$.

Since the two identities:
\[x_0 = 1 - \sum\limits_{i = 1}^m x_i, \quad\quad\quad\quad dx_0 = - \sum\limits_{i= 1}^m dx_i,\]
 hold on the affine standard simplex, we can assume $\eta = x_1^{k_1}\cdots x_m^{k_m} dx_{c_1}\wedge\dots\wedge dx_{c_l},$ with  $0 < c_1 < \dots < c_l.$ Once again the general case follows by linearity.

Following Construction \ref{h_m} we compute $\alpha_\eta$:
\[\alpha_\eta = (1-s)^{\sum\limits_{i = 1}^{m} k_i + l - 1} x_1^{k_1}\cdots x_m^{k_m}\left(\sum\limits_{i = 1}^l (-1)^{i} x_{c_i}ds\wedge dx_{c_1}\wedge\dots\wedge \widehat{dx_{c_i}}\wedge \dots\wedge dx_{c_l}\right),\]
and then $h_0(\eta)$ is:
\[h_0(\eta) = \frac{x_1^{k_1}\cdots x_m^{k_m}}{k_1 + \dots + k_m + l}\left(\sum\limits_{i = 1}^{l} (-1)^{i} x_{c_i}  dx_{c_1}\wedge\dots\wedge \widehat{dx_{c_i}}\wedge \dots\wedge dx_{c_l}\right).\]

Following Construction \ref{k_m} we have
\[\begin{split}\widehat f^\ast(\eta) =&  s^{\sum\limits_{i = 1}^{m} k_i + l - 1} x_1^{k_1}\cdots x_m^{k_m}\left(\sum\limits_{i = 1}^{l} (-1)^{i-1} x_{c_i} ds\wedge dx_{c_1}\wedge \dots\wedge\widehat{dx_{c_i}}\wedge \dots\wedge dx_{c_l}\right)\\
 &\; + s^{\sum\limits_{i = 1}^{m} k_i + l} x_1^{k_1}\cdots x_m^{k_m} dx_{c_1}\wedge \dots\wedge dx_{c_l}.\end{split}\]
Using Equation \eqref{nneqp}, it follows
\[\int\limits_{\widehat\Delta^0}s^{\sum\limits_{i = 1}^{m} k_i + l} = 0,\]
since we are integrating a 0-form the 1-simplex $\widehat\Omega_0$.
And consequently
\[\left(\int_{\widehat\Delta^0}\;\otimes \id\right) \left(s^{\sum\limits_{i = 1}^{m} k_i + l} x_1^{k_1}\cdots x_m^{k_m} dx_{c_1}\wedge\dots\wedge dx_{c_l}\right) = 0.\]
Thus to conclude the proof we have just to compute $k_{f_0}$.
\[\begin{split}k_{f_0}(\eta )&=
\left(\int_{\widehat\Delta^0}\;\otimes \id\right)\left(\widehat f^\ast(\eta)\right)\\ &=\frac{x_1^{k_1}\cdots x_m^{k_m}}{k_1 + \dots + k_m + l}\left(\sum\limits_{i = 1}^{l} (-1)^{i} x_{c_i}  dx_{c_1}\wedge\dots\wedge \widehat{dx_{c_i}}\wedge \dots\wedge dx_{c_l}\right).\end{split}\]
\end{proof}

\begin{lemma}\label{case1}
Fix an integer $n$, a function $f\in  I(n, m)$ and a polynomial differential form $\eta = x_0^{k_0}\cdots x_{f(n)}^0\cdots x_m^{k_m} dx_{c_1}\wedge\dots\wedge dx_{c_q}.$
It is not restrictive to assume that $c_1 < c_2 < \dots < c_q$ and $c_i \neq f(n), \text{ for all } i$. Assume, moreover, that we have \[\left|\{c_1,\dots, c_q\}\cap \text{Im}(f|_{[n-1]})\right| = n - 1,\quad\quad\text{and}\quad\quad  k_{f|_{[n - 1]}}(\eta) =  h_{f|_{[n - 1]}}(\eta).\]
Then it follows that $h_f(\eta) = k_f(\eta).$
\end{lemma}
\begin{proof}



The form $k_f(\eta)$ has negative degree, therefore it vanishes. It is not restrictive to assume that $n - 1\not\in C\cap \text{Im}(f|_{[n-1]})$. Thus $\eta$ has the form: 
\[\eta = x_0^{k_0}\cdots x_{f(n)}^0 \cdots x_m^{k_m} dx_{f(0)}\wedge\dots\wedge dx_{f(n - 2)}\wedge dx_{b_1}\wedge\dots\wedge dx_{b_l},\]
with $0 < b_1 < b_2 < \cdots < b_l< m$, $b_i\neq f(n - 1), f(n)$ and $l = q - n + 1 > 1$. Define the set
\[P = \{(p_0,\dots, p_{n -1})\in \mathbb N^{n } \text{ such that } 0\leq p_i \leq k_{f(i)},\; \forall i\}.\]
Next we compute $\widehat{f|_{[n - 1]}}^\ast(\eta).$ We have
\[\widehat{f|_{[n - 1]}}^\ast(\eta) =\sum\limits_{p\in P}\left( \eta_p\left(\sum\limits_{i = 1}^l x_{b_i}s^{l - 1}(-1)^{i+ n} d_{s,t_0,\dots,t_{n - 2},x_{b_1},\dots,\widehat{x_{b_i}},\dots,x_{b_l}}\right) + \omega_p\right),\]
where $\omega_p$ are forms which vanish under  $\mathlarger{\int\limits_{\widehat\Delta^{n-1}}\otimes \id}$ by a degree argument, and $\eta_p$ is the polynomial
\[\eta_p =\prod\limits_{i = 0}^{n-1}\binom{k_{f(i)}}{p_i}t_i^{p_i}x_i^{k_{f(i)} - p_i}s^{(\sum\limits_{i = 0}^m k_i - \sum\limits_{i = 0}^{n-1} p_i)} \prod\limits_{\tiny{\begin{array}{c}i = 0, \dots, m \\i\not\in f([n - 1])\end{array}}} x_i^{k_i}.\]
Thus we have that
\[h_{f|_{[n - 1]}}(\eta)=k_{f|_{[n - 1]}}(\eta) = \sum\limits_{p \in P} \eta_p\left(\sum\limits_{i = 1}^l x_{b_i}(-1)^{i} dx_{b_1}\wedge\dots\wedge\widehat{dx_{b_i}}\wedge\dots\wedge dx_{b_l}\right).\]
Then $\alpha_{k_{f|_{[n - 1]}}(\eta)}$ is equal to
\[\begin{split}
&\sum\limits_{p \in P} \left(\sum\limits_{i = 1}^l\sum\limits_{j < i} x_{b_i}x_{b_j}(-1)^{i + j} ds\wedge dx_{b_1}\wedge\dots\wedge\widehat{dx_{b_j}}\wedge\dots\wedge\widehat{dx_{b_i}}\wedge\dots\wedge dx_{b_l} \right.\\
&\quad + \left.\sum\limits_{i = 1}^l\sum\limits_{j > i}\right.\left. x_{b_i}x_{b_j}(-1)^{i + j - 1} ds\wedge dx_{b_1}\wedge\dots\wedge\widehat{dx_{b_i}}\wedge\dots\wedge\widehat{dx_{b_j}}\wedge\dots\wedge dx_{b_l}\vphantom{\sum_i^j}\right)\eta_ps^{l - 1}\\
&\qquad= 0.\end{split}\]
So we proved that $h_{n}\circ k_{f|_{[n - 1]}}(\eta) = h_f(\eta)= 0$; and this completes the first part of the proof.
\end{proof}

\begin{lemma}\label{case2}
Fix an integer $n$, a function $f\in  I(n, m)$ and a polynomial differential form $\eta = x_0^{k_0}\cdots x_{f(n)}^0\cdots x_m^{k_m} dx_{c_1}\wedge\dots\wedge dx_{c_q}.$
It is not restrictive to assume that $c_1 < c_2 < \dots < c_q$ and $c_i \neq f(n), \text{ for all } i$. Assume, moreover, that we have:\[\left|\{c_1,\dots, c_q\}\cap \text{Im}(f|_{[n-1]})\right| = n,\quad\quad\text{and}\quad\quad  k_{f|_{[n - 1]}}(\eta) =  h_{f|_{[n - 1]}}(\eta).\]
Then it follows that $h_f(\eta) = k_f(\eta).$
\end{lemma}
\begin{proof}
We can write $\eta$ as
\[\eta = x_0^{k_0}\cdots x_{f(n)}^0 \cdots x_m^{k_m} dx_{f(0)}\wedge\dots\wedge dx_{f(n - 1)}\wedge dx_{b_1}\wedge \dots\wedge dx_{b_l},\] 
with $0 < b_1 < b_2 < \dots < b_l< m$, $b_i\neq f(n)$ and $l = q - n \geq 1$. Consider, as in Lemma \ref{case1}, the set: 
\[P = \{(p_0, \dots, p_{n - 1})\in \mathbb N^n \text{ such that } 0\leq p_i \leq k_{f(i)},\; \forall i\}.\]
In order to compute $k_f(\eta)$, observe that:
\[\widehat{f}^\ast (\eta) = \sum\limits_{p\in P}\epsilon_p\theta_pd(sx_{f(0)} + t_0)\wedge\dots\wedge d(sx_{f(n - 1)} + t_{n - 1})\wedge d(sx_{b_1})\wedge\dots\wedge d(sx_{b_l}).\]
Where $\theta_p$ and $\epsilon_p$ are defined as: 
\[\epsilon_p = \prod\limits_{i = 0}^{n-1}\binom{k_{f(i)}}{p_i}x_{f(i)}^{k_{f(i)} - p_i}\prod\limits_{\tiny{\begin{array}{c}i = 0, \dots, m\\i\not\in f([n-1])\end{array}}} x_i^{k_i};\quad\quad\quad\quad \theta_p = s^{\sum\limits_{i = 0}^m k_i - \sum\limits_{i = 0}^{n-1} p_i}\prod\limits_{i = 0}^{n-1}t_i^{p_i}.\]
Then $k_f(\eta)$ is equal to: 
{\small\[\begin{split}
&\left(\int\limits_{\widehat\Delta^n} \otimes \id\right)
\left(\sum\limits_{p\in P}\left(\epsilon_p \theta_p\left(\sum\limits_{i = 1}^l x_{b_i}s^{l - 1}d_{t_0, \dots, t_{n - 1}, x_{b_1},\dots, x_{b_{i - 1}}, s, x_{b_{i + 1}},\dots, x_{b_l}}\right) + \omega_p\right)\right)\\
&\quad= \sum\limits_{p\in P}\frac{\left(\sum\limits_{i = 0}^m k_i - \sum\limits_{i = 0}^{n-1} p_i + l - 1\right)!\prod\limits_{i=0}^{n - 1} (p_i!)}{\left(\sum\limits_{i = 0}^m k_i + l + n\right)!}\epsilon_p\left(\sum\limits_{i = 1}^l(-1)^i x_{b_i}dx_{b_1}\wedge\dots\wedge\widehat{dx_{b_i}}\wedge\dots\wedge dx_{b_l}\right).\end{split}\]}
The forms $\omega_p$ vanish under  $\mathlarger{\int\limits_{\widehat\Delta^{n-1}}\otimes \id}$ by a degree argument. Moreover we have
\[\int\limits_{\widehat\Delta^n} \theta_p s^{l-1} dt_0\wedge\dots\wedge dt_n = \frac{\left(\sum\limits_{i = 0}^m k_i - \sum\limits_{i = 0}^{n-1} p_i + l - 1\right)!\prod\limits_{i=0}^{n - 1} (p_i!)}{\left(\sum\limits_{i = 0}^m k_i + l + n\right)!}.\]
 We can now compute $k_{f|_{[n - 1]}}(\eta)$. Recall that 
 \[\eta= x_0^{k_0}\dots x_{f(n)}^0 \dots x_m^{k_m} dx_{f(0)}\wedge\dots\wedge dx_{f(n - 1)}\wedge dx_{b_1}\wedge \dots\wedge dx_{b_l}.\]
 Thus we have:
 
\[\begin{split}&k_{f|_{[n - 1]}}(\eta)\\ 
&\;= \left(\int_{C^{n - 1}} \otimes \id\right)\left(\sum\limits_{p\in P}\epsilon_p \theta_p s^l\left(\vphantom{\int_i^n}dt_0\wedge\dots \wedge dt_{n - 1}\wedge dx_{b_1}\wedge\dots\wedge dx_{b_l}\right.\right. \\
&\quad\;\left.+ \sum\limits_{i = 0}^{n - 1}x_{f(i)}dt_0\wedge\dots\wedge dt_{i - 1}\wedge ds\wedge dt_{i + 1}\wedge \dots\wedge dt_{n - 1} \wedge dx_{b_1}\wedge \dots\wedge dx_{b_l}\right.\\
&\quad\;\left.+ \sum\limits_{i = 1}^l \sum\limits_{j = 0}^{n - 1}x_{b_i} d_{t_0,\dots,x_{f(j)},\dots,t_{n - 1},x_{b_1},\dots,x_{b_{i - 1}},s,x_{b_{i + 1}},\dots,x_{b_l}}\left.\vphantom{\int_i^n}\right) + \omega_p\right)\\
&\;=\sum\limits_{p\in P}\frac{\left(\sum\limits_{i = 0}^m k_i - \sum\limits_{i = 0}^{n-1} p_i + l\right)!\prod\limits_{i=0}^{n - 1} (p_i!)}{\left(\sum\limits_{i = 0}^m k_i + l + n\right)!} \epsilon_p \left(\vphantom{\sum_{i = 1}^l}dx_{b_1}\wedge\dots\wedge dx_{b_l}\right.\\
&\quad\;- \sum\limits_{i = 0}^{n - 1}x_{f(i)} dx_{b_1}\wedge\dots\wedge dx_{b_l}\\
&\quad\;\left.+ \sum_{i = 1}^l \sum\limits_{j = 0}^{n - 1}(-1)^{i - 1}x_{b_i} dx_{f(j)}\wedge dx_{b_1}\wedge\dots\wedge \widehat {dx_{b_i}}\wedge\dots\wedge dx_{b_l}\right).\end{split}\]
For the sake of readability call
\[\begin{split}\gamma_1 &= \sum\limits_{p\in P}\frac{\left(\sum\limits_{i = 0}^m k_i - \sum\limits_{i = 0}^{n-1} p_i + l\right)!\prod\limits_{i=0}^{n - 1} (p_i!)}{\left(\sum\limits_{i = 0}^m k_i + l + n\right)!} \epsilon_p dx_{b_1}\wedge\dots\wedge dx_{b_l};\\
\gamma_2 &= \sum\limits_{p\in P}\frac{\left(\sum\limits_{i = 0}^m k_i - \sum\limits_{i = 0}^{n-1} p_i + l\right)!\prod\limits_{i=0}^{n - 1} (p_i)!}{\left(\sum\limits_{i = 0}^m k_i + l + n\right)!} \epsilon_p\left(\sum\limits_{i = 0}^{n - 1} - x_{f(i)} dx_{b_1}\wedge \dots\wedge dx_{b_l} \right.\\
&\qquad\left.+ \sum\limits_{i = 1}^l \sum\limits_{j = 0}^{n - 1}(-1)^{i - 1}x_{b_i} dx_{f(j)}\wedge dx_{b_1}\wedge \dots\wedge \widehat{dx_{b_i}}\wedge\dots\wedge dx_{b_l}\right).\end{split}\]
Next we show that $h_{f(n)} (\gamma_1) = k_f(\eta)$ and $h_{f(n)} (\gamma_2) = 0$, which concludes the proof. The map $h_{f(n)}$ is the one described in Construction \ref{h_m}. The pullback of $\gamma_1$ under the map $f_{n}\colon [0]\to[m],\; 0\mapsto f(n)$ is:

\[f_n^\ast(\gamma_1) = \sum\limits_{p\in P} a_p (1 - s)^{\sum\limits_{i = 0}^m k_i - \sum\limits_{i = 0}^{n-1} p_i} \epsilon_p d((1 - s) x_{b_1})\wedge \dots\wedge d((1-s)x_{b_l}).\]
Where $a_p$ is defined for $p\in P$ as
\[a_p = \frac{\left(\sum\limits_{i = 0}^m k_i - \sum\limits_{i = 0}^{n-1} p_i + l\right)!\prod\limits_{i = 0}^{n - 1} (p_i!)}{\left(\sum\limits_{i = 0}^m k_i + l + n\right)!}.\]
Then we have:
\[\alpha_{\gamma_1} = \sum\limits_{p\in P} a_p \epsilon_p (1-s)^{\sum\limits_{i = 0}^m k_i - \sum\limits_{i = 0}^{n-1} p_i + l - 1} \left(\sum\limits_{i = 1}^l -(-1)^{i - 1}x_{b_i}  d_{s,x_{b_1},\dots,\widehat{x_{b_i}},\dots,x_{b_l}}\right).\]
By integration we get:
\[h_{f(n)}(\gamma_1) = \sum\limits_{p \in P} \tfrac{a_p}{\sum\limits_{i = 0}^m k_i - \sum\limits_{i = 0}^{n-1} p_i + l} \epsilon_p \left(\sum\limits_{i = 1}^l (-1)^{i} x_{b_i} dx_{b_1}\wedge\dots\wedge\widehat{dx_{b_i}}\wedge\dots \wedge dx_{b_l}\right).\]

To conclude the proof we need to show that $h_{f(n)}(\gamma_2) = 0$, but this follows from Lemma \ref{case1} since $n\not\in\{f(1), \dots, f(n - 1), b_1, \dots, b_l\}.$
\end{proof}

\begin{remark}A remarkable consequence of Theorem \ref{manettidupont} is that we have a simplicial contraction
\[\contr{\sC_\bullet}{\Omega_\bullet}{i_\bullet}{k_\bullet}{\pi_\bullet}\]

Getzler in \cite{getzler} showed that  $k_m^2 = 0$ and that $\pi_m k_m = 0$ (his proof of this latter fact works replacing $k_m$ with any family of functions satisfying Point $(2)$ of Theorem \ref{manettidupont}). This means that this is a simplicial contraction in the sense of \cite{eml} and of \cite{ls}.
\end{remark}

\end{document}